\documentclass[12pt, reqno]{amsart}
\usepackage{amsmath, amsthm, amscd, amsfonts, amssymb, graphicx, color}
\usepackage[bookmarksnumbered, colorlinks, plainpages]{hyperref}
\input{mathrsfs.sty}

\newtheorem{theorem}{Theorem}[section]
\newtheorem{lemma}[theorem]{Lemma}
\newtheorem{proposition}[theorem]{Proposition}
\newtheorem{corollary}[theorem]{Corollary}
\theoremstyle{definition}

\numberwithin{equation}{section}

\begin{document}

\title[Numerical radius inequalities]{Some generalized numerical radius inequalities for Hilbert space operators }

\author[M. Sattari, M. Moslehian, T. Yamazaki]{Mostafa Sattari$^1$, Mohammad Sal Moslehian$^1$ and Takeaki Yamazaki$^2$}

\address{$^1$ Department of Pure Mathematics, Center Of Excellence in Analysis on Algebraic Structures (CEAAS), Ferdowsi University of Mashhad, P. O. Box 1159, Mashhad 91775, Iran}
\email{moslehian@ferdowsi.um.ac.ir}
\email{Msattari.b@gmail.com}
\address{$^2$ Department of Electrical, Electronic and Computer Engineering Toyo University, Kawagoe-Shi, Saitama, 350-8585 Japan}
\email{e-mail: t-yamazaki@toyo.jp}



\begin{abstract}
We generalize several inequalities involving powers of the numerical radius for product of two operators acting on a Hilbert space.
For any $A, B, X\in \mathbb{B}(\mathscr{H})$ such that $A,B$ are positive, we establish some numerical radius inequalities for $A^\alpha XB^\alpha$ and $A^\alpha X B^{1-\alpha}\,\,(0 \leq \alpha \leq 1)$ and Heinz means under mild conditions.

\end{abstract} \maketitle

\section{Introduction}

Let $(\mathscr{H},\langle \cdot,\cdot\rangle )$ be a complex Hilbert space and $\mathbb{B}(\mathscr{H})$ denote the $C^*$-Algebra of all bounded linear operators on $\mathscr{H}$. An operator $A \in \mathbb{B}(\mathscr{H})$ is called positive if
$ \langle A x, x \rangle\geq0$ for all $x\in \mathscr{H}$. We write $A\geq 0$ if $A$ is positive.
The numerical radius of $A \in \mathbb{B}(\mathscr{H})$ is defined by
\begin{equation*}
w(A)= \sup \{ | \langle A x, x \rangle | : x \in \mathscr{H} , \| x \| =1 \}.
\end{equation*}
It is well known that $w(\cdot)$ defines a norm on $\mathbb{B}(\mathscr{H})$, which is equivalent to the usual operator norm $\| \cdot \|$. In fact, for any $A \in \mathbb{B}(\mathscr{H})$,
\begin{eqnarray}\label{1.1}
\frac{1}{2} \| A \| \leq w(A) \leq \| A \|.
\end{eqnarray}
Also if $A\in \mathbb{B}(\mathscr{H})$ is self-adjoint, then $w(A)= \| A \|.$\\
An important inequality for $w(A)$ is the power inequality stating that
\begin{equation*}
w(A^n)\leq w^n(A)
\end{equation*}
for $n=1,2,\ldots$\\
Several numerical radius inequalities improving the inequalities in \eqref{1.1} have been recently given in \cite{4,5,14}.

For instance, Dragomir proved that for any $A,B \in \mathbb{B}(\mathscr{H})$,
\begin{eqnarray} \label{4.1}
w^2 (A) \leq \frac{1}{2} \big( w(A ^2) + \| A\| ^2\big),
\end{eqnarray}
and
\begin{eqnarray} \label{5.1}
w^r (B ^* A ) \leq \frac{1}{2}\| (A^* A) ^r +( B^*B )^r \|
\end{eqnarray}
for all $r \geq 1$. The above inequalities can be found in \cite{5,3}, respectively. Some other interesting inequalities for numerical radius can be found in \cite{14,Mia,Y}.\\
In section \ref{eq22} of this paper, we first generalize inequalities \eqref{4.1} and \eqref{5.1}.\\ Our generalization of inequality \eqref{5.1}  in a particular case is sharper than this inequality.\\
In section \ref{eq33} we obtain numerical radius inequalities for Hilbert space operators $A^\alpha X B^\alpha$ and $A^\alpha X B^{1-\alpha}$ under conditions $A,B\geq0$ and $ 0 \leq \alpha \leq 1$. We also find a numerical radius inequality for Heinz means.

\section{Numerical radius inequalities for products of Two operators}
\label{eq22}

 To prove our generalized numerical radius inequalities, we need several well-known lemmas.
 The first lemma is a simple consequence of the classical Jensen and Young inequalities (see \cite{8}).
\begin{lemma} \label{le1}
 For $ a,b \geq0$, $0 \leq \alpha \leq 1$ and $p,q>1$ such that $\frac{1}{p}+\frac{1}{q}=1$
\begin{enumerate}
\item[(a)]
$ \displaystyle
a^\alpha b^{1-\alpha} \leq \alpha a +(1-\alpha)b \leq
\left[\alpha a^r +(1-\alpha) b^r\right]^{\frac{1}{r}} \quad\text{for}\quad r\geq 1$,
\item[(b)]
$\displaystyle
ab\leq\frac{a^p}{p}+\frac{a^q}{q}\leq \left(\frac{a^{pr}}{p}+\frac{b^{qr}}{q}\right)^{\frac{1}{r}}
\quad\text{for} \quad r\geq 1.$
\end{enumerate}
\end{lemma}
The second lemma follows from the spectral theorem for positive operators and Jensen's inequality (see \cite{11}).

\begin{lemma} [\textbf{McCarty inequality}] \label{le2}
Let $A \in \mathbb{B}(\mathscr{H})$, $ A \geq 0$ and let $x \in \mathscr{H}$ be any unit vector. Then
\begin{enumerate}
\item[(a)]
$\langle A x, x \rangle ^r \leq \langle A^r x, x \rangle \quad \text{for} \quad r \geq 1$,
\item[(b)]
$\langle A^r x , x \rangle \leq \langle Ax, x \rangle ^r \quad \text{for} \quad 0<r \leq 1$.
\end{enumerate}
\end{lemma}
The third lemma is known as the generalized mixed Schwarz inequality (see\cite{11}).
\begin{lemma} \label{le3}
Let $A \in \mathbb{B}(\mathscr{H})$ and $x,y \in \mathscr{H}$ be any vector.
\begin{enumerate}
\item[(a)]
If $ 0 \leq \alpha \leq 1$, then
$| \langle Ax, y \rangle | ^2 \leq \langle | A | ^{2 \alpha } x,x \rangle \langle | A^*| ^{2(1-\alpha) } y,y \rangle$,
\item[(b)]
If $f,g$ are nonnegative continuous functions on $[0, \infty) $ satisfying $f(t) g(t) =t $, $( t \geq 0)$, then
$| \langle A x,y \rangle | \leq \| f(|A| ) x \| \| g (| A^*|)y\|.$ \end{enumerate}
\end{lemma}

Now we are in a position to state the main result of this section.
First of all, we generalize inequality \eqref{4.1} for any $ r \geq 1$. We use some strategies used in \cite{5} to prove it.

\begin{theorem}
If $A \in \mathbb{B}(\mathscr{H})$, then
\begin{equation}\label{MOS1}
w^{2r} (A) \leq \frac{1}{2} \big( w^r(A^2) + \| A \| ^{2r} \big)
\end{equation}
for any $r \geq1$.
\end{theorem}

\begin{proof}
We recall the following refinement of the Cauchy--Schwarz inequality obtained by Dragomir in \cite{6}, see also . It says that
\begin{eqnarray} \label{7.1}
\| a \| \| b \| \geq | \langle a,b \rangle - \langle a,e \rangle \langle e, b \rangle | + | \langle a,e \rangle \langle e,b \rangle | \geq | \langle a,b \rangle |,
\end{eqnarray}
where $a,b,e$ are vectors in $\mathscr{H}$ and $ \| e \| =1$.\\
From inequality \eqref{7.1} we deduce that
\begin{equation*}
\frac{1}{2} \big( \| a \| \| b \| + | \langle a,b \rangle |) \geq | \langle a,e\rangle \langle e,b \rangle |.
\end{equation*}
Put $e= x$ with $\| x \| =1 ,~ a=Ax$ and $ b=A^* x $ in the above inequality and use Lemma \ref{le1} (a) to get
\begin{align*}
| \langle A x , x \rangle | ^2 \leq \frac{1}{2} \big( \| A x \| \| A^* x \|  + | \langle A^2 x,x \rangle |\big)\leq \left(
\frac{ \| A x \|^{r} \| A^* x \|^r + | \langle A^2 x, x \rangle | ^r} { 2} \right) ^{\frac{1}{r}},
\end{align*}
whence
\begin{eqnarray} \label{8.1}
| \langle A x,x \rangle | ^{2r} \leq \frac{1}{2} \big( \| Ax \| ^r \| A^* x \| ^r + | \langle A^2 x,x \rangle | ^r \big).
\end{eqnarray}
Taking the supremum over $ x \in \mathscr{H}$ with $\|x\|=1$ in inequality (\ref{8.1}) we obtain the desired inequality.
\end{proof}
The next result reads as follows.
\begin{proposition}\label{Prop2.5}
Let $A\in \mathbb{B}(\mathscr{H})$ and $f,g$ be nonnegative continuous functions on $[ 0, \infty) $ satisfying $f(t) g(t ) = t, \,\,   (t \geq 0)$. Then
\begin{eqnarray} \label{9.1}
w^{2r} (A) \leq \frac{1}{2} \left( \| A \| ^{2r} + \left\|\frac{1}{p} f^{pr} (|A^2|) +\frac{1}{q} g ^{qr} (|(A^2)^{*}|) \right\| \right).
\end{eqnarray}
 for all $r\geq 1$, $p\geq q>1$ with $\frac{1}{p}+\frac{1}{q}=1$ and $ qr \geq2$.
\end{proposition}
\begin{proof}
Let $ x \in \mathscr{H}$ be a unit vector. We have
\begin{align*}
| \langle A^2 x , x \rangle |^{r} & \quad \leq \| f (|A^2| ) x \|^{r} \| g( |(A^2)^{*}| ) x \|^{r} \quad \quad (\text{by Lemma \ref{le3} (b)} )\\
&\quad = \langle f^2 (|A^2| ) x, x \rangle ^{\frac{r}{2} }\langle g^2 (|(A^2)^{*}| ) x,x \rangle ^{\frac{r}{2}}\\
&\quad \leq \frac{1}{p} \langle f^2 (|A^2|) x,x \rangle ^{\frac{pr}{2}} +\frac{1}{q} \langle g^2 (|(A^2)^{*}|) x,x \rangle ^{\frac{qr}{2}}
 \quad\quad ( \text{by Lemma \ref{le1} (b)} )\\
& \quad \leq  \frac{1}{p}\langle f^{pr} (|A^2|) x,x \rangle +\frac{1}{q} \langle g^{qr} (|(A^2)^{*}|) x,x \rangle  \qquad\quad ( \text{by Lemma \ref{le2} (a)})\\
& = \left\langle \left(\frac{1}{p} f^{pr} (|A^2|) + \frac{1}{q} g^{qr} (|(A^2)^{*}|)\right) x,x \right\rangle.
\end{align*}
It follows from inequality \eqref{8.1} that
\begin{equation*}
| \langle A x , x \rangle |^{2r} \leq \frac{1}{2} \left( \| A x \| ^r \| A^* x \| ^r + \Big\langle \big(\frac{1}{p} f^{pr} (|A^2|) + \frac{1}{q} g^{qr} (|(A^2)^{*}|)\big) x,x \Big\rangle \right).
\end{equation*}
Taking the supremum over $ x \in \mathscr{H}$ with $\| x \| =1 $ in the above inequality we deduce the desired inequality \eqref{9.1}.
\end{proof}
Inequality \eqref{9.1} induces several numerical radius inequalities as special cases. For example the following result may be stated as well.

\begin{corollary}
If we take $f(t) = t^\alpha, g(t) = t^{1-\alpha}$ and $p=q=2$ in inequality \eqref{9.1}, then
\begin{equation*}
w^{2r} (A) \leq \frac{1}{2} \left( \| A \| ^{2r} +
\frac{1}{2} \| |A|^{4\alpha r} +|A^{*}|^{4(1-\alpha) r} \| \right)
\end{equation*}
for any $ r \geq1$ and $ 0 \leq \alpha \leq 1$.
\end{corollary}
\noindent In addition, by choosing $ \alpha = \frac{1}{2}$ we get $w^{2r} (A) \leq \| A \| ^{2r}$ for any $r \geq 1$, which is a generalization of the second inequality in \eqref{1.1}.

An operator $A$ on a Hilbert space $\mathscr{H}$ is said to be a paranormal operator if
\[
\| Ax\| ^2\leq \| A^2x\|.
\]
for any unit vector $x\in \mathscr{H}.$\\
Therefore if $A  \in \mathbb{B}(\mathscr{H})$ be a paranormal, then we get
\[
\| A^*A\| \leq \| A^2\|.
\]
On making use of above inequality and power inequality for numerical radius, we have the next result.
\begin{corollary}
If  $A\in \mathbb{B}(\mathscr{H})$ is paranormal and $f,g$ be as in Proposition \ref{Prop2.5}, then
\begin{equation*}
w^r (A) \leq \frac{1}{2} \left( \| A \| ^r + \left\|\frac{1}{p} f^{pr} (|A|) +\frac{1}{q} g ^{qr} (|A^*|) \right\| \right).
\end{equation*}
for all $r\geq 1$, $p\geq q>1$ with $\frac{1}{p}+\frac{1}{q}=1$ and $ qr \geq2$.
\end{corollary}
\begin{proof}
It is enough to use $w(A^2)\leq w^2(A)$ and then replacing $A^2$ by $A$ in the inequality \eqref{9.1}.
\end{proof}
The next result is an extension of \eqref{5.1} and is a kind of Young's inequality for operators. In fact in \cite{Mia}, the authors proved the following proposition:
\begin{proposition} \label{prop:8.2}
Suppose $A,B , X \in \mathbb{B}(\mathscr{H})$
and $f,g$ are nonnegative continuous functions on $[0, \infty) $ satisfying $f(t) g(t) =t $ for all $ t \geq 0$. Then
\begin{eqnarray} \label{12.1}
w^r( A X B ) \leq \big\| \frac{1}{p} \big[A f^2 (| X^{*} | ) A^{*} \big] ^{\frac{pr}{2}}
+\frac{1}{q}\big[ B^{*} g^2 (| X | ) B \big] ^{\frac{qr}{2}} \big\|
\end{eqnarray}
for all $r\geq 0$ and $p,q>1$ with
$\frac{1}{p}+\frac{1}{q}=1$ and $pr,qr\geq 2$.
\end{proposition}
They then deduce the following result:
\begin{proposition}\label{cor: w(AB)}
Let $A,B\in B(\mathscr{H})$. Then
\begin{align*}
w^{r}(B^{*}A) & \leq  \| \frac{1}{p} |A|^{pr}+\frac{1}{q}|B|^{qr}\|
\end{align*}
holds for all $r\geq 0$ and $p,q>1$ with
$\frac{1}{p}+\frac{1}{q}=1$ and $pr, qr\geq 2$.
\end{proposition}
Proposition \ref{cor: w(AB)} is given by putting
$X=I$ and $f(t)=g(t)=\sqrt{t}$ in Proposition
\ref{prop:8.2}.
Now we show that Propositions \ref{prop:8.2} and
\ref{cor: w(AB)} are equivalent.
To see this, we have only to prove Proposition \ref{prop:8.2}
from Proposition \ref{cor: w(AB)}.

\begin{proof}[Proof of Proposition \ref{prop:8.2}]
Let $X=U|X|$ be the polar decomposition of $X$.
Put $S=f(|X|)U^{*}A^{*}$ and $T=g(|X|)B$.
Then by Proposition \ref{cor: w(AB)}, we have
\begin{align*}
& w^{r}(S^{*}T)  \leq  \| \frac{1}{p} |S|^{pr}+\frac{1}{q}|T|^{qr}\| \\
\Longleftrightarrow
& \quad
w^r( A X B ) \leq \big\| \frac{1}{p} \big[A f^2 (| X^{*} | ) A^{*} \big] ^{\frac{pr}{2}} +
\frac{1}{q}\big[ B^{*} g^2 (| X | ) B \big] ^{\frac{qr}{2}} \big\|.
\end{align*}
\end{proof}

The following theorem gives
an upper bound for $w(B^{*}A)$.

\begin{theorem}\label{thm: w(AB) and w(BA)}
Let $A,B\in \mathbb{B}(\mathscr{H})$. Then
\begin{align*}
w^{r}(B^{*}A) & \leq \frac{1}{4} \| (AA^{*})^{r}+(BB^{*})^{r}\| +
\frac{1}{2}w^{r}(AB^{*})
\end{align*}
for all $r\geq 1$.
\end{theorem}

By Theorem \ref{thm: w(AB) and w(BA)} and inequality \eqref{5.1}, we have

\begin{align*}
w^{r}(B^{*}A)  \leq \frac{1}{4} \| (AA^{*})^{r}+(BB^{*})^{r}\| +\frac{1}{2}w^{r}(AB^{*}) \leq \frac{1}{2}\| (AA^{*})^{r}+(BB^{*})^{r}\|.
\end{align*}
Hence if both $A$ and $B$ are normal operators, then Theorem \ref{thm: w(AB) and w(BA)}
is a sharper inequality than \eqref{5.1}.
To prove Theorem \ref{thm: w(AB) and w(BA)}, we use the famous polarization identity as follows:
\begin{equation}
 \langle x,y\rangle =\frac{1}{4}\sum_{k=0}^{3}
\|x+i^{k}y\|^{2}i^{k}
\label{polarization identity}
\end{equation}
holds for all $x,y \in \mathscr{H}$.

\begin{proof}[Proof of Theorem \ref{thm: w(AB) and w(BA)}]
First of all, we note that
\begin{equation}
 w(T) =\sup_{\theta\in \mathbb{R}} \| \mbox{Re}(e^{i\theta}T)\|,
\label{eq:real part}
\end{equation}
where $\mbox{Re}X$ means the real part of an operator
$X$, i.e., $\mbox{Re}X=\frac{X+X^{*}}{2}$.
Because
$$ | \langle Tx,x\rangle |= \sup_{\theta\in \mathbb{R}} \mbox{Re} \{e^{i\theta}
\langle Tx,x\rangle\} $$
and
$$ \sup_{\theta\in \mathbb{R}}
\| \mbox{Re} (e^{i\theta}T)\| =
\sup_{\theta\in \mathbb{R}}w(\mbox{Re} (e^{i\theta}T))=
w(T). $$

\medskip

For a unit vector $x\in \mathscr{H}$,
we have
\begin{align*}
 \mbox{Re}\langle e^{i\theta}B^{*}A x,x\rangle
&
 = \mbox{Re}\langle e^{i\theta}A x, B x\rangle \\
&
= \frac{1}{4} \| (e^{i\theta}A+B)x\|^{2}-
\frac{1}{4} \| (e^{i\theta}A-B)x\|^{2}
& \mbox{(by \eqref{polarization identity})}\\
&
\leq  \frac{1}{4} \| (e^{i\theta}A+B)x\|^{2}\\
&
\leq  \frac{1}{4} \| e^{i\theta}A+B\|^{2}\\
& = \frac{1}{4} \| e^{-i\theta}A^{*}+B^{*}\|^{2}
& \mbox{(by $\|X^{*}\|=\|X\|$)} \\
&
= \frac{1}{4} \| (e^{-i\theta}A^{*}+B^{*})^{*}(e^{-i\theta}A^{*}+B^{*})\|
& \mbox{(by $\|X\|^{2}=\|X^{*}X\|$)}\\
&
= \frac{1}{4} \| AA^{*}+BB^{*}+
e^{i\theta}AB^{*}+e^{-i\theta}BA^{*}\| \\
&
\leq
 \frac{1}{4} \| AA^{*}+BB^{*} \|+
\frac{1}{2}\|\mbox{Re} (e^{i\theta}AB^{*})\| \\
&
\leq
 \frac{1}{4} \| AA^{*}+BB^{*} \|+
\frac{1}{2} w(AB^{*}) &\mbox{(by \eqref{eq:real part})}.
\end{align*}
Now taking the supremum over $x \in \mathscr{H}$ with $\|x\|=1$ in the above inequality produces
$$ w(B^{*}A)\leq \frac{1}{4}\| AA^{*}+BB^{*}\| +\frac{1}{2}w(AB^{*}).$$

For $r\geq 1$, since $t^{r}$ and $t^{\frac{1}{r}}$
are convex and operator concave functions, respectively, we have
\begin{align*}
w^{r}(B^{*}A)
& \leq
\left( \frac{1}{2} \left\|\frac{AA^{*}+BB^{*}}{2}\right\| +\frac{1}{2}w(AB^{*})\right)^{r}\\
& \leq
\frac{1}{2} \left\|\frac{AA^{*}+BB^{*}}{2}\right\|^{r} +\frac{1}{2}w^{r}(AB^{*})\\
& \leq
\frac{1}{2} \left\|\left(\frac{(AA^{*})^{r}+(BB^{*})^{r}}{2}\right)^{\frac{1}{r}}\right\|^{r} +\frac{1}{2}w^{r}(AB^{*})\\
& =
\frac{1}{2} \left\|\frac{(AA^{*})^{r}+(BB^{*})^{r}}{2}\right\| +\frac{1}{2}w^{r}(AB^{*}).
\end{align*}
Hence Theorem \ref{thm: w(AB) and w(BA)} is proven.
\end{proof}

The next corollary is an extension of an inequality
shown in \cite[Theorem 2.1]{Y}.

\begin{corollary}\label{Cor2.11}
Let $T\in B(\mathcal{H})$ and $T=U|T|$ be the
polar decomposition of $T$, and let
$\tilde{T}(\alpha)=|T|^{\alpha}U|T|^{1-\alpha}$ be the
generalized Aluthge transformation of $T$. Then we have
$$ w^{r}(T)\leq \frac{1}{4} \|
|T|^{2r\alpha}+|T|^{2r(1-\alpha)}\|+
\frac{1}{2}w^{r}(\tilde{T}(\alpha)) $$
holds for $r\geq 1$.
\end{corollary}

\begin{proof}
Put $A=|T|^{\alpha}$ and $B=|T|^{1-\alpha}U^{*}$ in Theorem \ref{thm: w(AB) and w(BA)}. Then we have
\begin{align*}
& w^{r}(B^{*}A)  \leq \frac{1}{4} \| |A^{*}|^{2r}+|B^{*}|^{2r}\| +
\frac{1}{2}w^{r}(AB^{*})\\
\Longleftrightarrow
&  \quad
w^{r}(T)  \leq \frac{1}{4} \| |T|^{2r\alpha}+|T|^{2r(1-\alpha)}\| +
\frac{1}{2}w^{r}(\tilde{T}(\alpha)).
\end{align*}
\end{proof}

\section{Numerical Radius Inequalities for Product of operators} \label{eq33}

The main purpose of this section is to find upper bounds for $A^\alpha X B^\alpha$ and $ A^\alpha X B^{1-\alpha}$ for the case when $0 \leq \alpha \leq 1$.
Also we find a numerical radius inequality for Heinz means.

The following theorem gives us a new bound for powers of the numerical radius.

\begin{theorem}\label{thm:3.1}
Suppose $A, B , X \in \mathbb{B}(\mathscr{H})$ such that $A, B $ are positive. Then
\begin{equation*}
w^{r} (A^\alpha X B^\alpha ) \leq \| X \| ^{r} \| \frac{1}{p} A^{pr} + \frac{1}{q} B ^{qr} \|^\alpha
\end{equation*}
for all $ 0 \leq \alpha \leq 1$, $ r \geq 0$ and $p,q>1$
with $\frac{1}{p}+\frac{1}{q}=1$
and $pr,qr\geq 2$.
\end{theorem}
\begin{proof}
For any unit vector $x \in \mathscr{H}$ and by the Cauchy--Schwarz  inequality we have
\begin{align*}
| \langle A^{\alpha} XB^{\alpha} x, x \rangle |^{r}
& = | \langle X B^{\alpha} x, A^{\alpha} x \rangle |^{r} \\
& \leq \| X B^{\alpha} x \| ^{r} \| A^{\alpha} x \| ^{r} \\
& \leq \| X \| ^{r} \langle A^{2\alpha } x,x \rangle^{\frac{r}{2}} \langle B^{2\alpha} x,x \rangle^{\frac{r}{2}} \\
& \leq \| X \| ^{r}\left(
\frac{1}{p} \langle A^{2\alpha} x,x \rangle^{\frac{pr}{2}}+
\frac{1}{q} \langle B^{2\alpha} x,x \rangle ^{\frac{qr}{2}} \right)
& (\text{by Lemma \ref{le1} (b)})\\
& \leq
\| X \|^{r}
\left(\frac{1}{p}\langle A^{pr} x,x \rangle ^{\alpha} +
\frac{1}{q} \langle B^{qr} x,x \rangle ^{\alpha} \right)
& \mbox{(by Lemma \ref{le2} )}\\
& \leq
\| X \|^{r}
\left(\frac{1}{p}\langle A^{pr} x,x \rangle  +
\frac{1}{q} \langle B^{qr} x,x \rangle  \right)^{\alpha}
& \mbox{(by the concavity of $t^{\alpha}$)}\\
& =
\| X \|^{r}
\langle
\left(\frac{1}{p} A^{pr}   +
\frac{1}{q} B^{qr}\right) x,x \rangle^{\alpha}.
\end{align*}
Now by taking the supremum over $x \in \mathscr{H}$ with $\|x\|=1$ in the above inequality we infer that
Theorem \ref{thm:3.1}.
\end{proof}

\begin{corollary}
Let $\tilde{T}=|T|^{\frac12}U|T|^{\frac12}$ be the
 Aluthge transformation of $T$ such that $U$ is partial isometry. Then
\begin{equation*}
w(\tilde{T})\leq \| T \|.
\end{equation*}
\end{corollary}
\begin{proof}
If we take $r=1$, $\alpha =\frac12$, $p=q=2$, $A=B=|T| $ and $X=U$ in Theorem \ref{thm:3.1}, then
\[
w(\tilde{T})\leq \|\frac12|T|^2+\frac12|T|^2 \|^{\frac12}= \||T|^2 \|^{\frac12}= \| T \|.
\]
\end{proof}

Our next result  is to find an upper bound for power of the numerical radius of $A^\alpha X B^{1-\alpha} $ under assumption $0\leq \alpha \leq 1$.

\begin{theorem}\label{thm:3.3}
Suppose $A, B, X \in \mathbb{B}(\mathscr{H})$ such that $A$, $B$ are positive. Then
\begin{equation*}
w^{r} (A^\alpha X B^{1-\alpha}) \leq \|X\|^{r} \|\alpha A^r +(1-\alpha)B^r \|
\end{equation*}
for all $r\geq 2$ and $0\leq \alpha \leq 1$.
\end{theorem}

\begin{proof}
Let $x\in \mathscr{H}$ be a unit vector. Then
\begin{align}
|\langle A^\alpha X B^{1-\alpha} x,x \rangle |^r
& =
| \langle XB^{1-\alpha} x , A^\alpha x \rangle |^r
 \nonumber \\
& \leq  \|X \|^r \|B^{1-\alpha} x \|^r \|A^\alpha x\|^r
 \nonumber \\
& =  \|X \|^r  \langle B^{2(1-\alpha)}x,x \rangle^{\frac{r}{2}}
 \langle A^{2\alpha} x,x \rangle^{\frac{r}{2}}
 \nonumber  \\
& \leq
\|X \|^r  \langle A^{r} x,x \rangle^{\alpha} \langle B^{r}x,x \rangle^{1-\alpha}  \quad\qquad\text{(by\ Lemma\ \ref{le2})} \nonumber \\
& \leq  \|X \|^r
\langle \left( \alpha A^{r} +(1-\alpha)  B^{r}\right) x,x \rangle  \quad \text{(by\ Lemma\ \ref{le1}(a))} \label{eq:Heinz mean}
\end{align}
Taking the supremum over $x\in \mathscr{H}$ with $\|x\|=1$ in the above inequality we deduce the desired inequality.
\end{proof}
Noting that our inequality in previous theorem is a generalization of the second inequality \eqref{1.1} when we set $A=B=I$.\\
Now assume that $A, B, X \in \mathbb{B}(\mathscr{H})$.
the Heinz means for matrices are defined by
\begin{equation*}
H_{\alpha}(A,B)= \frac{A^\alpha XB^{1-\alpha}+A^{1-\alpha}XB^\alpha}{2}
\end{equation*}
in which $0\leq\alpha\leq1$ and $A,B\geq0$, see \cite{m2}.

Our final result in this section is to find  a numerical radius inequality for Heinz means. For this purpose, we use Theorem \ref{thm:3.3} and the convexity of function $f(t)=t^r (r\geq1).$

\begin{theorem}\label{thm 3.4}
Suppose $A, B, X \in \mathbb{B}(\mathscr{H})$ such that $A$, $B$ are positive. Then
\begin{align*}
w^{r}(A^{\frac{1}{2}}XB^{\frac{1}{2}}) & \leq
w^r\big( \frac{A^\alpha XB^{1-\alpha}+
A^{1-\alpha}XB^\alpha}{2}\big) \\
& \leq
\|X\|^{r} w\left(\frac{A^{r}+B^{r}}{2}\right) \\
& \leq
\frac{\|X\|^r}{2}\Big( \big\|\alpha A^r +(1-\alpha)B^r \big\|+
\big\|(1-\alpha) A^r +\alpha B^r \big\| \Big).
\end{align*}
for all $r\geq 2$ and $0\leq \alpha \leq 1$.
\end{theorem}

To prove Theorem \ref{thm 3.4}, we need the following
lemma.

\begin{lemma}\label{lem: 3.5}
Let $A, B \in \mathbb{B}(\mathscr{H})$ be invertible
self-adjoint operators and $X\in \mathbb{B}(\mathscr{H})$.
Then
\begin{align*}
w (X)
& \leq
w\left( \frac{AXB^{-1}+A^{-1}XB}{2}\right).
\end{align*}
\end{lemma}

\begin{proof}
First of all, we shall show the case $A=B$ and
$X$ is self-adjoint.
Let $\lambda \in \sigma(X)$. Then
$$
\lambda \in \sigma (X)=\sigma (AXA^{-1})
\subseteq \overline{W(AXA^{-1})}.
$$
Since $\lambda \in \mathbb{R}$, we have
$$ \lambda=\mbox{Re}\lambda \in
\mbox{Re} \overline{W(AXA^{-1})}=
\overline{W(\mbox{Re}(AXA^{-1}))}. $$
So we obtain
$$ w(X)=r(X)\leq w(\mbox{Re}(AXA^{-1}))=
w\left( \frac{AXA^{-1}+A^{-1}XA}{2}\right). $$

\medskip

Next we shall show this lemma for arbitrary
$X\in \mathbb{B}(\mathscr{H})$ and
invertible self-adjoint operators $A$ and $ B$.
Let $\hat{X}=\begin{pmatrix} 0 & X \\ X^{*} & 0
\end{pmatrix}$
and $\hat{A}=\begin{pmatrix} A & 0 \\ 0 & B
\end{pmatrix}$.
Then $\hat{X}$ and $\hat{A}$ are self-adjoint. Hence we have
$$ w(\hat{X})\leq
w\left( \frac{\hat{A}\hat{X}\hat{A}^{-1}+
\hat{A}^{-1}\hat{X}\hat{A}}{2}\right). $$
Here $w(\hat{X})=w(X)$ and
\begin{align*}
w\left( \frac{\hat{A}\hat{X}\hat{A}^{-1}+
\hat{A}^{-1}\hat{X}\hat{A}}{2}\right)
& =
\frac{1}{2} w(
\begin{pmatrix}
0 & AXB^{-1}+A^{-1}XB \\
BX^{*}A^{-1}+B^{-1}X^{*}A & 0
\end{pmatrix} )\\
& =
\frac{1}{2} w(AXB^{-1}+A^{-1}XB  ).
\end{align*}
Therefore we obtain the desired inequality.
\end{proof}

\begin{proof}[Proof of Theorem \ref{thm 3.4}.]
We may assume that $A$ and $B$ are invertible.
By Lemma \ref{lem: 3.5}, we have
\begin{align*}
 w^{r}(A^{\frac{1}{2}}XB^{\frac{1}{2}})
& \leq
w^{r} \left(\frac{A^{\alpha-\frac{1}{2}} \cdot
A^{\frac{1}{2}}XB^{\frac{1}{2}}\cdot B^{\frac{1}{2}-\alpha}+
A^{\frac{1}{2}-\alpha} \cdot
A^{\frac{1}{2}}XB^{\frac{1}{2}}\cdot B^{\alpha-\frac{1}{2}}}{2}\right)\\
& =
w^{r}\left(\frac{A^{\alpha} XB^{1-\alpha}+
A^{1-\alpha}XB^{\alpha}}{2}\right).
\end{align*}
On the other hand, by inequality \eqref{eq:Heinz mean}, for any $r\geq 2$ we have
$$|\langle A^\alpha X B^{1-\alpha} x,x \rangle |^r
 \leq  \|X \|^r
\langle \left( \alpha A^{r} +(1-\alpha)  B^{r}\right) x,x \rangle .
$$
Hence we have
\begin{align*}
& |\langle
\frac{A^\alpha X B^{1-\alpha}+
A^{1-\alpha} X B^{\alpha}}{2}x,x\rangle|^{r}\\
& \leq
\left( \frac{|\langle
A^\alpha X B^{1-\alpha}x,x\rangle| +
|\langle A^{1-\alpha} X B^{\alpha}x,x\rangle|}{2}\right)^{r} \\
& \leq
 \frac{|\langle
A^\alpha X B^{1-\alpha}x,x\rangle|^{r} +
|\langle A^{1-\alpha} X B^{\alpha}x,x\rangle|^{r}}{2}
\quad (\mbox{ by the convexity of $t^{r}$})\\
& \leq
\frac{\|X \|^{r}}{2}
\left\{
\langle \left( \alpha A^{r} +(1-\alpha)  B^{r}\right) x,x \rangle
+
\langle \left( (1-\alpha) A^{r} +\alpha  B^{r}\right) x,x \rangle\right\}\\
& =
\|X \|^r
\langle  \frac{A^{r} +  B^{r}}{2} x,x \rangle.
\end{align*}
Therefore we obtain
\begin{align*}
&  w^{r} \left(\frac{A^{\alpha}XB^{1-\alpha}+
A^{1-\alpha}XB^{\alpha}}{2}\right) \\
& \leq
\|X\|^{r} w\left(\frac{A^{r}+B^{r}}{2}\right) \\
& \leq
\frac{\|X\|^{r}}{2} \left(
w(\alpha A^{r}+(1-\alpha)B^{r})+
w((1-\alpha) A^{r}+\alpha B^{r})\right)\\
& =
\frac{\|X\|^{r}}{2} \left(
\| \alpha A^{r}+(1-\alpha)B^{r}\|+
\|(1-\alpha) A^{r}+\alpha B^{r}\| \right).
\end{align*}
\end{proof}

\bigskip
\noindent
{\bf Acknowledgment.} The authors would like to
express our cordial thanks to Professor Fumio Hiai
to giving us a crucial comment. 

\bigskip
\bibliographystyle{amsplain}

\end{document}